\documentclass[11pt,a4paper,reqno]{amsart}
\usepackage[utf8]{inputenc}
\usepackage[T1]{fontenc}       
\usepackage{ae}   
\usepackage{amsfonts}         
\usepackage{amsmath}
\usepackage{amsthm}
\usepackage{amssymb}
\usepackage{mathtools}
\usepackage{calc}
\usepackage{centernot}
\usepackage{color}
\usepackage[pdftex]{hyperref}
\usepackage{mathrsfs}

\usepackage{bigints} 

\usepackage{mathpazo} 

\usepackage{xargs}

\usepackage{todonotes}

\newcommandx{\huom}[2][1=]{\todo[linecolor=red,backgroundcolor=red!10,bordercolor=red,#1]{#2}}

\newcommand{\ve}{\varepsilon}

\newcommand{\R}{\mathbb{R}}

\newcommand{\HH}{\mathbb{H}}

\newcommand{\ang}[1]{\left\langle #1 \right\rangle}

\newcommand{\p}{\partial}

\DeclareMathOperator\dist{dist}
\DeclareMathOperator\dv{div}

\DeclareMathOperator\tR{Ric}
\DeclareMathOperator\Ric{Ric}

\DeclareMathOperator\Hess{Hess}

\numberwithin{equation}{section}

\theoremstyle{plain}
\newtheorem{thm}{Theorem}[section]
\newtheorem{lem}[thm]{Lemma}

\newtheorem{prop}[thm]{Proposition}

\theoremstyle{definition}

\newtheorem{exa}[thm]{Example}
\newtheorem{rem}[thm]{Remark}

             {\begin{list}{\arabic{enumi}.}{\usecounter{enumi}%
              \setlength{\labelsep}{0.5em}%
              \settowidth{\labelwidth}{\arabic{enumi}.}%
              \setlength{\leftmargin}{\labelwidth+\labelsep}}}%
             {\end{list}}

\author[J.-B. Casteras et al.]{Jean-Baptiste Casteras}
\address{J.-B. Casteras,
Department of Mathematics and Statistics, P.O. Box 68 (Pietari Kalmin katu 5),
 00014 University of Helsinki, 
  Finland.
}
\email{jeanbaptiste.casteras@gmail.com}

\author[]{Esko Heinonen}
\address{E. Heinonen,
Department of Mathematics and Statistics,
P.O. Box 35, 40014 University of Jyv\"askyl\"a, Finland.
}
\email{ea.heinonen@gmail.com}

\author[]{Ilkka Holopainen}
\address{I. Holopainen,
Department of Mathematics and Statistics, P.O. Box 68 (Pietari Kalmin katu 5),
 00014 University of Helsinki, 
  Finland.
}
\email{ilkka.holopainen@helsinki.fi}

\author[]{Jorge H. De Lira}
\address{J. H. De Lira,
  Departamento de Matem\'atica,
  Universidade Federal do Cear\'a, Bloco 914, Campus do Pici,
  Fortaleza, Cear\'a, 60455-760, Brazil.
}
\email{jorge.lira@mat.ufc.br}


\title[Mean curvature flow with prescribed contact angle]{Non-parametric mean curvature flow with prescribed contact angle in Riemannian products}
\date{\today}

\begin{document}

\begin{abstract} Assuming that there exists a translating soliton $u_\infty$ with speed $C$ in a domain $\Omega$ and with prescribed contact angle on $\partial\Omega$, we prove that a graphical solution to the mean curvature flow with the same prescribed contact angle converges to $u_\infty +Ct$ as $t\to\infty$.  We also generalize the recent existence result of Gao, Ma, Wang and Weng to 
non-Euclidean settings under suitable bounds on convexity of $\Omega$ and Ricci curvature in $\Omega$. 
\end{abstract}
\maketitle


\section{Introduction}
We study a non-parametric mean curvature flow in a Riemannian product $N\times\R$ represented by 
graphs 
\begin{equation}\label{evolgraphs}
M_t:=\left\lbrace\big(x,u(x,t)\big)\colon x\in\bar{\Omega}\right\rbrace
\end{equation}
with prescribed contact angle with the cylinder $\partial\Omega\times\R$. 

We assume that $N$ is a Riemannian manifold and $\Omega\Subset N$ is a relatively compact domain with smooth boundary $\partial\Omega$. We denote by $\gamma$ the inward pointing unit normal vector field to 
$\partial\Omega$.
The boundary condition is determined by a given smooth function $\phi\in C^{\infty}(\partial\Omega)$, 
with $|\phi|\le\phi_0<1$, and the initial condition by a smooth function $u_0\in C^{\infty}(\bar{\Omega})$.

The function $u$ above in \eqref{evolgraphs} is a solution to the following evolution equation 
\begin{equation}\label{para}
\begin{cases}
\dfrac{\partial u}{\partial t} = W
\dv \dfrac{\nabla u}{W}& 
\text{ in } \Omega\times [0,\infty),
\\
\dfrac{\partial_\gamma u}{W}:=\dfrac{\ang{\nabla u,\gamma}}{W}=\phi& \text{ on }\p\Omega\times [0,\infty),\\
u(\cdot,0)=u_0& \text{ in }\bar{\Omega},
\end{cases}
\end{equation}
where $W=\sqrt{1+|\nabla u|^2}$ and $\nabla u$ denotes the gradient of $u$ with respect to the Riemannian metric on $N$ at $x\in \bar{\Omega}$.
The boundary condition above can be written as
\begin{equation}\label{bcondi}
\ang{\nu,\gamma} = \phi,
\end{equation} 
where $\nu$ is the downward pointing unit normal to the graph of $u$, i.e.
	\[
	\nu(x) = \frac{\nabla u(x,\cdot)-\partial_t}{\sqrt{1+|\nabla u(x,\cdot)|^2}},\ x\in\bar{\Omega}.
	\]

The longtime existence of the solution $u_t:=u(\cdot,t)$ to \eqref{para} and convergence as $t\to\infty$ have been studied under various conditions on $\Omega$ and $\phi$.
Huisken \cite{huiskenJDE} proved the existence of a smooth solution in a $C^{2,\alpha}$-smooth bounded domain 
$\Omega\subset\R^n$ for $u_0\in C^{2,\alpha}(\bar{\Omega})$ and $\phi\equiv 0$. Moreover, he showed that $u_t$ converges to a constant function as $t\to\infty$.
In \cite{AW} Altschuler and Wu complemented Huisken's results for prescribed contact angle in case $\Omega$ is a smooth bounded strictly convex domain in $\R^2$. Guan \cite{guan} proved a priori gradient estimates and established longtime existence of solutions in case $\Omega\subset\R^n$ is a smooth bounded domain. Recently, Zhou \cite{ZhouIMRN} studied mean curvature type flows in a Riemannian product $M\times\R$ and proved the longtime  existence of the solution for relatively compact smooth domains $\Omega\subset M$. Furthermore, he extended the convergence result of Altschuler and Wu to the case $M$ is a Riemannian surface with nonnegative curvature and $\Omega\subset M$ is a smooth bounded strictly convex domain; see \cite[Theorem 1.4]{ZhouIMRN}. 

The key ingredient, and at the same time the main obstacle, for proving the uniform convergence of $u_t$ has been a difficulty to obtain a time-independent gradient estimate. We circumvent this obstacle by modifying the method of Korevaar \cite{kore-capillary}, Guan \cite{guan} and Zhou \cite{ZhouIMRN} and obtain a uniform gradient estimate in an arbitrary relatively compact smooth domain $\Omega\subset N$ provided there exists a translating soliton with speed $C$ and with 
the prescribed contact angle condition \eqref{bcondi}.

Towards this end, let $d$ be a smooth bounded function defined in some neighborhood of $\bar{\Omega}$ such that $d(x)= \min_{y\in\p\Omega} \dist(x,y)$, the distance to the boundary $\p\Omega$, for points $x\in\Omega$ sufficiently close to $\p\Omega$. Thus $\gamma=\nabla d$ on $\p\Omega$. 
We assume that $0\le d\le 1$, $|\nabla d|\le 1$ and $|\Hess d| \le C_d$ in $\bar{\Omega}$.
We also assume that the function $\phi\in C^{\infty}(\p\Omega)$ is extended as a smooth function to the whole $\bar\Omega$, satisfying the condition $|\phi| \le \phi_0<1$.

Our main theorem is the following:  

\begin{thm}\label{thm-main}
Suppose that there exists a solution $u_\infty$ to the translating soliton equation
\begin{equation}\label{elli}
\begin{cases}
\dv \dfrac{\nabla u_\infty}{\sqrt{1+|\nabla u_\infty|^2}} =\dfrac{C_\infty}{\sqrt{1+|\nabla u_\infty|^2}}& 
\text{ in } \Omega,
\\
\dfrac{\partial_\gamma u_\infty}{\sqrt{1+|\nabla u_\infty|^2}}=\phi& \text{ on }\p\Omega,
\end{cases}
\end{equation}
where $C_\infty$ is given by
\begin{equation}\label{defC}
C_\infty=\dfrac{-\int_{\p\Omega}\phi\,\mathrm{d}\sigma}{\int_\Omega \left(1+|\nabla u_\infty|^2\right)^{-1/2}\,\mathrm{d}x}.
\end{equation}
Then the equation \eqref{para} has a smooth solution $u\in C^\infty (\bar{\Omega},[0,\infty ))$ with 
$W\leq C_1$, where $C_1$ is a constant depending on $\phi$, $u_0$, $C_d$, and the Ricci curvature of 
$\Omega$. Moreover, $u(x,t)$ converges uniformly to $u_\infty (x) +C_\infty t$ as $t\to\infty$.
\end{thm}

Notice that the existence of a solution $u\in C^\infty\big(\bar{\Omega}\times [0,\infty)\big)$ to \eqref{para}  is given by \cite[Corollary 4.2]{ZhouIMRN}.

\begin{rem}
Very recently, Gao, Ma, Wang, and Weng \cite{gmww} proved the existence of such $u_\infty$ and obtained Theorem~\ref{thm-main} for smooth, bounded, strictly convex domains $\Omega\subset\R^n$ for sufficiently small $|\phi|$; see \cite[Theorem 1.1, Theorem 3.1]{gmww}. It turns out that their proof can be generalized beyond the Euclidean setting under suitable bounds on the convexity of $\Omega$ and the Ricci curvature in $\Omega$.
\end{rem}
More precisely, let $\Omega\Subset N$ be a relatively compact, strictly convex domain with smooth boundary admitting a smooth defining function $h$ such that $h<0$ in $\Omega$, 
 $h=0$ on $\partial\Omega$, 
 \begin{equation}\label{k1}
 \big(h_{i;j}\big)\ge k_1\big(\delta_{ij}\big)
 \end{equation} 
 for some constant $k_1>0$ and $\sup_{\Omega}|\nabla h|\le 1$, $h_\gamma =-1$ and $|\nabla h|=1$ on $\partial\Omega$.
 Furthermore, by strict convexity of $\Omega$, the second fundamental form of $\partial\Omega$ satisfies
\begin{equation}\label{k0}
\big(\kappa_{ij}\big)_{1\le i,j\le n-1}
\ge\kappa_0 \big(\delta_{ij}\big)_{1\le i,j\le n-1},
\end{equation}
where $\kappa_0>0$ is the minimal principal curvature of $\partial\Omega$. In the Euclidean case, $N=\R^n$, such functions $h$ are constructed in \cite{cns}. We give some simple examples at the end of Section 3.
\begin{thm}\label{thm-cor}
Let $\Omega\Subset N$ be a smooth, strictly convex, relatively compact domain associated with constants $k_1>0$ and $\kappa_0>0$ as in \eqref{k1} and \eqref{k0}.
Let $\alpha<\min\{\kappa_0,k_1(n-1)/2\}$ and assume that the Ricci curvature in $\Omega$ satisfies 
$|\Ric|<\alpha(k_1(n-1)-\alpha)/(n+1)$. Then there 
exists $\varepsilon_0>0$ 
such that if $\phi=:\cos\theta\in C^3(\bar{\Omega})$ satisfies $|\cos\theta|\le\varepsilon_0\le 1/4$ and $||\nabla\theta||_{C^1(\bar{\Omega})}\le \varepsilon_0$ in $\bar{\Omega}$, there exist a unique constant $C_\infty$ and a solution $u_\infty$ to \eqref{elli}. Furthermore, $u_\infty$ is unique up to an additive constant.
\end{thm}
We will sketch the proof of Theorem~\ref{thm-cor} in Section~\ref{sketch}.

\section{Proof of Theorem~\ref{thm-main}}
Let $u$ be a solution to \eqref{para} in $\bar{\Omega}\times\R$.
Given a constant $C_\infty\in\R$ we define, following the ideas of Korevaar \cite{kore-capillary}, Guan \cite{guan} and Zhou \cite{ZhouIMRN}, a function $\eta\colon \bar{\Omega}\times\R\to (0,\infty)$ by setting
\begin{equation}\label{defeta}
\eta = e^{K(u-C_\infty t)}\left(Sd + 1 - \frac{\phi}{W} \ang{\nabla u, \nabla d}\right),
\end{equation}
where $K$ and $S$ are positive constants to be determined later. We start with a gradient estimate.
\begin{prop}\label{propgradest}
Let $u$ be a solution to \eqref{para} and define $\eta$ as in \eqref{defeta}. Then, for a fixed $T>0$, letting
$$(W \eta)(x_0 ,t_0)=\max_{x\in \bar{\Omega},\ t\in [0,T] } (W\eta)(x,t), $$
there exists a constant $C_0$ only depending on $C_d$, $\phi$, $C_\infty$, and the lower bound for the Ricci curvature in $\Omega$ such that $W(x_0 ,t_0)\leq C_0$.
\end{prop}

\begin{proof}
Let $g=g_{ij}dx^idx^j$ be the Riemannian metric of $N$. We denote by $(g^{ij})$ the inverse of $(g_{ij})$, $u_j =\partial u/\partial x^j$,
and $u_{i;j}=u_{ij}-\Gamma^{k}_{ij}u_k$. We set
\[
	a^{ij} = g^{ij} - \frac{u^iu^j}{W^2}
	\]
and define an operator $L$ by $Lu = a^{ij}u_{i;j}-\partial_t u$. Observe that \eqref{para} can be rewritten as $Lu=0$. 
In all the following, computations will be done at the maximum point $(x_0,t_0)$ of $\eta W$. 
We first consider the case where $x_0\in\partial\Omega$. We choose normal coordinates at $x_0$ such that $g_{ij}=g^{ij} = \delta^{ij}$ at $x_0$,  $\p_n = \gamma$, 
	\[
	u_1 \ge 0, \quad u_i = 0 \quad \text{for } \, 2\le i \le n-1.
	\]
This implies that
	\[
	d_i = 0 \, \text{ for } 1 \le i \le n-1, \, \, d_n = 1, \, \text{ and } \,  d_{i;n} = 0 \, \text{ for } 1\le i \le n.
	\]
We have
	\begin{align}\label{Weta-nDeriv}
	0 &\ge (W\eta)_n = W_n \eta + W\eta_n \nonumber \\
	&= e^{K(u-C_\infty t)} \Big( SW_nd + W_n - \frac{\phi W_n}{W} g^{ij}u_id_j + SWd_n - \frac{W}{W} \phi_n g^{ij} u_id_j \nonumber \\
	&\qquad - \frac{W}{W} \phi g^{ij}(u_{i;n}d_j + u_i d_{j;n}) + W \frac{W_n}{W^2}\phi g^{ij}u_id_j \nonumber \\
	&\qquad + KW u_n (Sd + 1- \frac{\phi}{W} g^{ij} u_id_j)  \Big) \nonumber \\	
	&= e^{K(u-C_\infty t)} \Big( W_n + SW -  \phi_n u_n  - \phi u_{n;n}  + KWu_n ( 1- \phi^2) \Big).
	\end{align}
Using our coordinate system, we get
	\begin{align*}
	0 &\ge \frac{W_n}{W} + S -  \frac{\phi_n u_n}{W}  -\frac{ \phi u_{n;n}}{W}  + Ku_n ( 1- \phi^2) \\	
	&= S -\frac{u_1^2 d_{1;1}}{W^2} +\frac{u_1\phi_1}{W} \Big( 1+ \frac{2\phi^2}{1-\phi^2} \Big)	
	-\frac{\phi u_1}{W} K u_1 \\
	&\qquad -\frac{\phi_n u_n}{W}  + K u_n ( 1- \phi^2) \\
	&\ge  S - C - \frac{K \phi u_1^2}{W} + Ku_n( 1- \phi^2)   \\
	&= S - C - \frac{K\phi}{W} \ge S - C - \frac{K}{W},
	\end{align*}
for some constant $C$ depending only on $C_d$ and $\phi$. So choosing $S \ge C +1$, we get that 
\begin{equation}\label{wbound1}
W(x_0,t_0) \le K.
\end{equation}

Next we assume that $x_0\in\Omega$ and that $S\ge C+1$, where $C$ is as above. Let us recall from 
\cite[Lemma 3.5]{ZhouIMRN} that 
	\begin{equation*}\label{LW-identity}
	LW =  \frac{2}{W} a^{ij} W_i W_j + \tR(\nu_N,\nu_N) W + |A|^2W,
	\end{equation*}
where $\nu_N=\nabla u/W$ and $|A|^2=a^{ij}a^{\ell k}u_{i;k}u_{j;\ell}/W^2$ is the squared norm of the second fundamental form of the graph $M_t$.
Since $0 = W_i\eta + W\eta_i$, for every $i=1,\ldots,n$, we deduce that
	\begin{align*}
	0\ge L(W\eta) &= W L\eta + \eta\Big(LW - 2a^{ij}\frac{W_iW_j}{W} \Big) \\
	&= WL\eta + \eta W\left( |A|^2 + \tR(\nu_N,\nu_N) \right).
	\end{align*}
This yields to
	\begin{equation}\label{est-eta-Ric}
	\frac{1}{\eta}L\eta + |A|^2 + \tR(\nu_N,\nu_N)  \le 0.
	\end{equation}
To simplify the notation, we set
	\[
	h = Sd+1-\phi u^k d_k/W=Sd+1-\phi \nu^k d_k.
	\]
So we have
	\begin{equation}\label{L-eta}
	\frac{1}{\eta} L\eta
	= K^2 a^{ij} u_i u_j + K L(u-C_\infty t) +  \frac{2K}{h} a^{ij} u_i h_j +  \frac{1}{h} Lh.
	\end{equation}

We can compute $Lh$ as
	\begin{align*}
	Lh &= a^{ij} \big( S d_{i;j} - (\phi d_k)_{i;j} \nu^k - (\phi d_k)_i \nu^k_{\ j} -(\phi d_k)_j \nu^k_{\ i} - \phi d_k L\nu^k \big) \nonumber \\
	&\ge - C - 2 a^{ij} (\phi d_k)_i \nu^k_{\ j} - \phi d_k L \nu^k.
	\end{align*}
Since, by \cite[Lemma 3.5]{ZhouIMRN},
	\begin{equation*}\label{L-nu-2nd}
	L\nu^k =  \tR(a^{k\ell}\p_\ell,\nu_N) - |A|^2 \nu^k
	\end{equation*}
and, by Young's inequality for matrices, 
	\[
	a^{ij} (\phi d_k)_i \nu^k_{\ j} = \frac{1}{W} (\phi d_k)_i a^{ij}a^{\ell k} u_{\ell;j} \le \frac{|A|^2}{6} + C,
\]
we get the estimate 
	\begin{equation}\label{Lh-2nd}
	Lh \ge - C -   |A|^2/3 +  \phi d_k\nu^k |A|^2 
	\end{equation}
by using the assumption that $\tR$ is bounded.

Next we turn our attention to the other terms in \eqref{L-eta}. We have
	\begin{equation}\label{a-uf-deriv}
	a^{ij}u_i = \frac{u^j}{W^2} \quad \text{and} \quad a^{ij}u_i u_j 	= 1-\frac{1}{W^2}.
	\end{equation}
Then we note that by the assumptions, we clearly have
	\begin{equation}\label{L-uf}
	KL(u-C_\infty t) =KC_\infty\ge -KC,
	\end{equation}
and we are left to consider
	\begin{align}\label{a-uh-2nd}
	a^{ij} u_i h_j &= \frac{u^j h_j}{W^2} = \frac{u^j\big(Sd_j - (\phi d_k)_j \nu^k - \phi d_k \nu^k_{\ j}\big)}{W^2} \nonumber \\
	&\ge -C - \frac{\phi d_k u^j \nu^k_{\ j}}{W^2}  \nonumber \\
	&= -C + \frac{K\phi a^{\ell k} d_k u_\ell}{W}  + \frac{\phi}{hW} a^{\ell k} d_kh_\ell \nonumber \\
&= -C + \frac{K\phi a^{\ell k} d_k u_\ell}{W}  \nonumber \\
&+\frac{S\phi a^{\ell k} d_k d_\ell}{hW} -  \frac{\phi a^{\ell k} d_k (\phi d_s)_\ell \nu^s}{hW} -  \frac{\phi^2 a^{\ell k} d_k d_s a^{sm} u_{m;\ell}}{hW^2} \nonumber \\
& \geq -C - \frac{CK}{W^2} -\frac{|A|^2}{3K}.
	\end{align}

Plugging the estimates \eqref{Lh-2nd}, \eqref{a-uf-deriv}, \eqref{L-uf}, and \eqref{a-uh-2nd} into \eqref{L-eta} and using \eqref{est-eta-Ric} with the Ricci lower bound we obtain
\begin{align*}
0&\ge K^2 \left( 1- \frac{1}{W^2}  \right) - CK - \frac{2K}{h} \left( C + \frac{CK}{W} + \frac{CK}{W^2} + \frac{|A|^2}{3K} \right)  \\
&\quad - \frac{1}{h} \big( C + |A|^2/3 - \phi d_k\nu^k |A|^2  \big) + |A|^2 -C\\
&= K^2 \left( 1- \frac{1}{W^2}  - \frac{C}{hW^2}\right) - KC \left( 1 + \frac{1}{h}  \right) - \frac{|A|^2}{h} \\
&\quad + \frac{\phi d_k\nu^k |A|^2}{h} - \frac{C}{h} + |A|^2 -C.
\end{align*}
Then collecting the terms including $|A|^2$ and 
noticing that
	\[
	1 - \frac{1}{h} + \frac{\phi d_k\nu^k}{h} 
	= \frac{Sd}{h} \ge 0
	\]
we have
\begin{align*}
0 \ge K^2 \left( 1- \frac{1}{W^2} - \frac{C}{hW^2} \right) - CK \left( 1 + \frac{1}{h}  \right) - C.
\end{align*}	
Now choosing $K$ large enough, we obtain $W(x_0,t_0) \le C_0$, where $C_0$ depends only on $C_\infty$, $d$, $\phi$, the lower bound of the Ricci curvature in $\Omega$, and the dimension of $N$. We notice that the constant $C_0$ is independent of $T$.
\end{proof}
Since 
\[
e^{K\big(u(\cdot,t)-C_\infty t\big)}(1-\phi_0)\le \eta
\le e^{K\big(u(\cdot,t)-C_\infty t\big)}(S+2),
\]
we have 
\begin{align}\label{Wbound}
W(x,t)&\le \dfrac{(W\eta)(x_0,t_0)}{\eta(x,t)}\nonumber
\\
&\le \dfrac{C_0\eta(x_0,t_0)}{\eta(x,t)}\\
&\le \frac{C_0(S+2)}{1-\phi_0} e^{K\big(u(x_0,t_0)-C_\infty t_0-u(x,t)+C_\infty t\big)}\nonumber
\end{align}
for every $(x,t)\in \bar{\Omega} \times [0,T]$.

We observe that the function 
$u_\infty(x)+C t$ solves the equation \eqref{para} with the initial condition $u_0=u_\infty$ if $u_\infty$ is a solution to the elliptic equation \eqref{elli} and $C$ is given by \eqref{defC}.
As in \cite[Corollary 2.7]{AW}, 
applying a parabolic maximum principle (\cite{lieber}) we obtain:
\begin{lem}\label{lem2.7} Suppose that \eqref{elli} admits a solution $u_\infty$ with the unique constant $C$ given by \eqref{defC}.
Let $u$ be a solution to \eqref{para}. Then, we have
$$|u(x,t)-Ct|\leq c_2,$$
for some constant $c_2$ only depending on $u_0,\ \phi$, and $\Omega$. 
\end{lem}

\begin{proof}
Let $V(x,t)=u(x,t)-u_\infty (x) $, where $u_\infty$ is a solution to \eqref{elli}. We see that $V$ satisfies
$$
\begin{cases}\dfrac{\partial V}{\partial t} =\tilde{a}^{ij} V_{i;j}+b^i V_i +C &\text{ in }\Omega\times [0,T)\\
 \tilde{c}^{ij}V_i \nu_j =0 &\text{ on }\partial\Omega\times [0,T),  \end{cases}
 $$
where $\tilde{a}^{ij}$, $\tilde{c}^{ij}$ are positive definite matrices and $b^i \in \R$. Then the proof of the lemma follows by applying the maximum principle.
\end{proof}

In view of Lemma~\ref{lem2.7}, taking $C_\infty=C$, and observing that the constant $C_0$ is independent of $T$, we get from \eqref{Wbound} a uniform gradient bound.
\begin{lem}\label{unigradest} Suppose that \eqref{elli} admits a solution $u_\infty$ with the unique constant $C$ given by \eqref{defC}.
Let $u$ be a solution to \eqref{para}. Then 
$W(x,t)\leq C_1$ for all $(x,t)\in\bar{\Omega}\times [0,\infty)$ with a constant $C_1$ depending only on 
$\phi_0,\ u_0$, and $\Omega$.
\end{lem}

Having a uniform gradient bound in our disposal, applying once more the strong maximum principle for linear uniformly parabolic equations, we obtain:
\begin{thm}\label{thm-conv}
Suppose that \eqref{elli} admits a solution $u_\infty$ with the unique constant $C$ given by \eqref{defC}.
Let $u_1$ and $u_2$ be two solutions of \eqref{para} with the same prescribed contact angle as $u_\infty$. Let $u=u_1 -u_2$. Then $u$ converges to a constant function as $t\rightarrow \infty$. In particular, if 
 $C$ is given by \eqref{defC}, then $u_1 (x,t) - u_\infty (x) -Ct$  converges uniformly to a constant as $t\to\infty$.
\end{thm}

\begin{proof}
The proof is given in \cite[p. 109]{AW}. We reproduce it for the reader's convenience. 
One can check that $u$ satisfies
$$\begin{cases} \dfrac{\partial u}{\partial t} =\tilde{a}^{ij} u_{i;j}+b^i u_i &\text{ in } \Omega \times [0,\infty)\\ \tilde{c}^{ij}u_i \nu_j =0 &\text{ on }  \partial \Omega \times [0,\infty),\end{cases}$$
where $\tilde{a}^{ij}$, $\tilde{c}^{ij}$ are positive definite matrices and $b^i \in \R$. By the strong maximum principle, we get that the function $F_u(t)=\max u(\cdot,t) - \min u(\cdot,t) \geq 0$ is either strictly decreasing or $u$ is constant. Assuming on the contrary that $\lim_{t\to\infty}u$ is not a constant function, setting $u_n (\cdot,t)=u(\cdot,t-t_n)$ for some sequence $t_n \rightarrow \infty$, we would get a non-constant solution, say $v$, defined on $\Omega \times (-\infty ,+\infty)$ for which $F_v$ would be constant. We get a contradiction with the maximum principle.
\end{proof}

Theorem~\ref{thm-main} now follows from Lemma~\ref{unigradest} and Theorem~\ref{thm-conv}.

\section{Proof of Theorem~\ref{thm-cor}}\label{sketch}
Theorem~\ref{thm-cor} is essentially proven in \cite[Theorem 2.1, 3.1]{gmww}. The only extra ingredient we must take into account in our non-flat case is the following Ricci identity for the Hessian $\varphi_{i;j}$ of a smooth function $\varphi$
\begin{equation}\label{ricid}
\varphi_{k;ij}=\varphi_{i;kj}=\varphi_{i;jk}+R^\ell_{kji}\varphi_\ell.
\end{equation} 
For the convenience of the reader, we mostly use the same notations as in \cite{gmww}. Thus let $h$ be a smooth defining function of $\Omega$ such that $h<0$ in $\Omega$, $h=0$ on 
$\partial \Omega$, $(h_{i;j}) \geq k_1 (\delta_{ij})$ for some constant $k_1>0$ and $\sup_{\Omega} |\nabla h|\leq 1$, $h_\gamma =-1$ and $|\nabla h|=1$ on $\partial \Omega$.
Furthermore, by strict convexity of $\Omega$, the second fundamental form of $\partial\Omega$ satisfies
\[
(\kappa_{ij})_{1\le i,j\le n-1}
\ge\kappa_0 (\delta_{ij})_{1\le i,j\le n-1},
\]
where $\kappa_0>0$ is the minimal principal curvature of $\partial\Omega$.

We consider the equation
\begin{equation}\label{capil-eq}
	\begin{cases}
	a^{ij}u_{i;j}:=\left(g^{ij}-\frac{u^i u^j}{1+|\nabla u|^2}\right) u_{i;j}
 = \ve u &\text{ in } \Omega \\
	\p_\gamma u = \phi \sqrt{1+|\nabla u|^2}&\text{ on } \p\Omega
	\end{cases}
	\end{equation}
for small $\ve>0$. Writing $\phi=-\cos\theta$, $v=\sqrt{1+|\nabla u|^2}$ and 
\[
\Phi(x)=\log w(x) +\alpha h(x),
\]
where $w(x)=v-u^\ell h_\ell \cos \theta$ and $\alpha>0$ is a constant to be determined, we assume that the maximum of $\Phi$ is attained in a point 
$x_0\in\bar{\Omega}$. If $x_0\in\partial\Omega$, 
we can proceed as in \cite[pp. 34-36]{gmww}. 
Thus choosing
$0<\alpha<\kappa_0$ and $0<\varepsilon_0\le \varepsilon_\alpha<1$ such that
\begin{equation}\label{eps-cond}
\kappa_0-\alpha>\frac{\varepsilon_\alpha(M_1+3)}{1-\varepsilon_\alpha^2},
\end{equation}
where $M_1=\sup_{\bar{\Omega}}|\nabla^2 h|$, yields an upper bound
\[
|\nabla'u(x_0)|^2\le \dfrac{\frac{\varepsilon_0(M_1+3)}{1-\varepsilon_0^2}+\alpha}{\kappa_0-\alpha-\frac{\varepsilon_0(M_1+3)}{1-\varepsilon_0^2}}
<
\dfrac{\kappa_0}{\kappa_0-\alpha-\frac{\varepsilon_\alpha(M_1+3)}{1-\varepsilon_\alpha^2}}
\]
for the tangential component of $\nabla u$ on $\partial\Omega$. Combining this with the boundary condition
$u_\gamma=-v\cos\theta$ gives an upper bound for 
$|\nabla u(x_0)|$ and hence for $\Phi(x_0)$.

The only difference to the Euclidean case occurs when $x_0\in\Omega$, i.e. is an interior point of $\Omega$.
At this point we have, using the same notations as in \cite[p. 42]{gmww}, 
\[
0=\Phi_i(x_0)=\frac{w_i}{w}+\alpha h_i
\]
and
\[
0\ge a^{ij}\Phi_{i;j}(x_0)=\frac{a^{ij}w_{i;j}}{w}-\alpha^2 a^{ij}h_i h_j +\alpha a^{ij}h_{i;j}=:I+II+III.
\]
We choose normal coordinates at $x_0$ such that 
$u_1(x_0)=|\nabla u(x_0)|$ and\\ $(u_{i;j}(x_0))_{2\le i,j\le n}$ is diagonal. Then at $x_0$, we have
\[
II+III \geq -\alpha^2 (1+1/v^2)+\alpha k_1 (n-1+1/v^2).
\]
We denote $J=a^{ij}w_{i;j}=J_1 +\tilde{J}_2+J_3+J_4$, where $J_1,J_3$ and $J_4$ are as in \cite[(2.19)]{gmww}. We have, by \cite[(2.22)]{gmww},
\[
J_3+J_4\geq - C (|\cos \theta|+|\nabla\theta|+|\nabla^2 \theta|)u_1 - C (|\cos \theta| +|\nabla\theta|)\sum_{i=2}^n |u_{ii}|,
\]
where $C$ depends only on $n, M_1$ and $\sup_{\bar{\Omega}}|\nabla^3h|$.
Writing $S^{\ell} = \frac{u_\ell}{v}- h_\ell \cos \theta$ and using the Ricci identity 
\[
a^{ij}u_{k;ij}=a^{ij}u_{i;jk}+\Ric (\partial_k , \nabla u)
\] 
(see \cite[(2.28)]{ZhouIMRN}) and 
\eqref{capil-eq}, we get
\begin{align*}
\tilde{J}_2&= a^{ij}\left(\frac{u^k u_{k;ij}}{v}- u_{k;ij}h^k \cos \theta\right)= S^k a^{ij}u_{i;jk} +S^k \Ric (\partial_k , \nabla u)\\
&=-S^k a^{ij}_{;k}u_{i;j}+ S^k (\varepsilon u)_k+S^k \Ric (\partial_k , \nabla u)\\
& =J_2+\varepsilon u_1 S^1 +S^k \Ric (\partial_k , \partial_1)|\nabla u|,
\end{align*}
where $J_2$ is as in \cite[(2.19)]{gmww}.
Since $|S^1|\le 2$ and $|S^k|\le 1$ for $k\ge 2$, we obtain
\begin{equation}\label{j2est}
\tilde{J}_2\ge J_2 - (n+1)|\Ric_{\Omega}| |\nabla u|,
\end{equation}
where $|\Ric_{\Omega}|$ is the bound for the Ricci curvature in 
$\Omega$, i.e. $|\Ric(x)|\le |\Ric_{\Omega}|$ for all unit vectors $x\in T\Omega$. 
At this point, we can proceed as in \cite{gmww} to get that
\[
J_1+J_2 \geq \sum_{i=2}^n \frac{u_{ii}^2}{2v}.
\]
So combining the previous estimates, we find
\[
I= \frac{J}{w}\geq -C ( |\cos \theta| +|\nabla\theta|+|\nabla^2 \theta|)-(n+1)|\Ric_{\Omega}|.
\]
Hence we obtain
\begin{align*}
0 &\geq I+II+III \ge
 -C ( |\cos \theta| +|\nabla\theta|+|\nabla^2 \theta|)-(n+1)|\Ric_{\Omega}| \\
 &\qquad -\alpha^2 (1+1/v^2)+\alpha k_1 (n-1+1/v^2)\\
&=: C_1 + C_2/v^2,
\end{align*}
where
\[
C_1=-C\varepsilon_0 -(n+1)|\Ric_{\Omega}|+\alpha\big(k_1(n-1)-\alpha\big)
\]
and $C_2=\alpha(k_1-\alpha)$. If $C_1>0$ and $C_2>0$, we get a contradiction, and therefore the maximum of $\Phi$ is attained on $\partial\Omega$. 
If $C_1>0$ and $C_2<0$, then $v^2\le -C_2/C_1$ and again we have an upper bound for $\Phi(x_0)$.
To have $C_1>0$ we need 
\begin{equation}\label{ric-cond}
|\Ric_{\Omega}|< \big(\alpha (k_1 (n-1) - \alpha\big) -C \varepsilon_0)/(n+1).
\end{equation} 
Fixing $\alpha<\min\{\kappa_0,k_1(n-1)/2\}$ and assuming that 
\begin{equation}\label{ric-cond2}
|\Ric_{\Omega}| < \big(\alpha (k_1 (n-1) - \alpha\big)/(n+1)
\end{equation}
and, finally, choosing $0<\varepsilon_0\le \min\{\varepsilon_\alpha,1/4\}$ small 
enough so that \eqref{ric-cond} holds,
we end up again with a contradiction, and therefore 
the maximum of $\Phi$ is attained on $\partial\Omega$. 
All in all, we have obtained a uniform gradient bound for a solution $u$ to \eqref{capil-eq} that is independent of $\varepsilon$.
Once the uniform gradient bound is established the rest of the proof goes as in \cite{AW} (or \cite{gmww}).

In some special cases we get sharper estimates than those above.
\begin{exa}
As the first example let us consider the hyperbolic space $\HH^n$ and a geodesic ball $\Omega=B(o,R)$. Furthermore,  we choose
\[
h(x)=\frac{r(x)^2}{2R}- \frac{R}{2}
\]
as a defining function for $\Omega$. Here $r(\cdot)=d(\cdot,o)$ is the distance to the center
$o$. Then $\kappa_0=\coth R$ and we may choose $k_1=1/R$.
Since $\Ric(\partial_k,\partial_1)=-(n-1)\delta_{k1}$, \eqref{j2est} can be replaced by
\[
\tilde{J}_2\ge J_2 - 2(n-1)|\nabla u|
\]
and consequently \eqref{ric-cond2} can be replaced by
\[
2(n-1)<\alpha\big((n-1)/R-\alpha\big),
\]
where $\alpha<\min\{\coth R,\tfrac{n-1}{2R}\}$. Hence we obtain an upper bound for the radius $R$. For instance,
if $n=2$, then $\alpha<\tfrac{1}{2R}$ and we need
$R<\tfrac{1}{2\sqrt{2}}$. For all dimensions, $\alpha=1$ and $R<\tfrac{n-1}{2n-1}$ will do. 
\end{exa}
\begin{exa}
As a second example let $N$ be a Cartan-Hadamard manifold with sectional curvatures bounded from below by $-K^2$,
with $K>0$.
Again we choose $\Omega=B(o,R)$ and 
\[
h(x)=\frac{r(x)^2}{2R}-\frac{R}{2}.
\]
Now $1/R\le\kappa_0\le K\coth(KR)$ and again we may choose $k_1=1/R$.
This time $\Ric(\partial_1,\partial_1)\ge -(n-1)K^2$ and $\Ric(\partial_k,\partial_1)\ge -\tfrac{1}{2}(n-1)K^2$
for $k=2,\ldots,n$, and therefore instead of \eqref{j2est} and \eqref{ric-cond2} we have
\[
\tilde{J}_2\ge J_2 - K^2\big((n+1)^2/2-2\big)|\nabla u|
\]
and
\[
K^2\big((n+1)^2/2 -2\big)<\alpha\big((n-1)/R-\alpha\big),
\]
where $\alpha<\min\{1/R,\tfrac{n-1}{2R}\}$. Again we obtain upper bounds for the radius $R$. If $n\ge 3$ we need
\[
R<\left(\frac{n-2}{K^2\big((n+1)^2/2-2\big)}\right)^{1/2}
\]
whereas for $n=2$ the bound
\[
R<\frac{1}{2\sqrt{2}K}
\]
is enough since now $\Ric(\partial_2,\partial_1)=0$.
\end{exa}

\subsubsection*{Conflict of interest:}
  Authors state no conflict of  interest.

\end{document}